\documentclass[11pt]{article}%
\usepackage[letterpaper, portrait, margin=1in, headheight=0pt]{geometry}
\usepackage{mathtools}
\usepackage{enumerate}
\usepackage{amsmath,amsthm,amssymb,amsfonts}
\usepackage{enumitem}
\usepackage{mathrsfs}
\usepackage{tikz-cd}
\usepackage{multicol}
\usepackage[title]{appendix}
\usetikzlibrary{quotes,angles}

\DeclareMathOperator{\ad}{ad}

\DeclareMathOperator{\Der}{Der}

\DeclareMathOperator{\nil}{nil}

\newcommand{\F}{\mathbb{F}}

\newcommand{\inv}{^{-1}}
\newcommand{\vp}{\varphi}
\newcommand{\G}{\Gamma}

\newcommand{\alg}{\mathscr{P}}

\newcommand{\vv}{_{\vdash}}
\newcommand{\dd}{_{\dashv}}

\newtheorem{thm}{Theorem}[section]
\newtheorem{lem}[thm]{Lemma}
\newtheorem{prop}[thm]{Proposition}
\newtheorem{cor}[thm]{Corollary}

\theoremstyle{definition}
\newtheorem{defn}{Definition}
\newtheorem{ex}[subsubsection]{Example}

\theoremstyle{remark}

\title{Extensions of Nilpotent Algebras}
\author{Erik Mainellis}
\date{}

\begin{document}

\maketitle

\begin{abstract}
Given a pair of nilpotent Lie algebras $A$ and $B$, an extension $0\xrightarrow{} A\xrightarrow{} L\xrightarrow{} B\xrightarrow{} 0$ is not necessarily nilpotent. However, if $L_1$ and $L_2$ are extensions which correspond to lifts of a map $\Phi:B\xrightarrow{} \text{Out}(A)$, it has been shown that $L_1$ is nilpotent if and only if $L_2$ is nilpotent. In the present paper, we prove analogues of this result for the algebras of Loday. As an important consequence, we thereby gain its associative analogue as a special case of diassociative algebras.  
\end{abstract}

\section{Introduction}
Let $A$ and $B$ be nilpotent Lie algebras. In \cite{yankosky}, Bill Yankosky proved that the nilpotency of an extension $0\xrightarrow{} A\xrightarrow{} L\xrightarrow{} B\xrightarrow{} 0$ depends on $A$, $B$, and a map $\Phi:B\xrightarrow{} \text{Out}(A)$. In particular, given a pair of extensions $L_1$ and $L_2$ corresponding to lifts of $\Phi$, $L_1$ is nilpotent if and only if $L_2$ is nilpotent. This result was based on the work of James A. Schafer, who proved the group analogue in \cite{schafer}.

Beyond Lie algebras, the objective of the present paper is to prove Schafer’s result for six other types of algebras. Loday introduced three of these (Zinbiel, diassociative, and dendriform algebras \cite{loday cup product, loday dialgebras}) and generated interest in another (Leibniz algebras). The remaining two types are associative and commutative algebras. We observe that Yankosky’s work rests on the assumptions of nonabelian 2-cocycles, also called factor systems, which have long been known in the context of Lie algebras. As discussed in \cite{mainellis}, factor systems are a tool for working on the extension problem of algebraic structures. The work herein is an application of factor systems, which were developed for all seven types of algebras in \cite{mainellis}.

For the sake of this paper, it suffices to prove the Leibniz and diassociative cases. Indeed, as discussed in the introduction of \cite{mainellis}, any result which holds for the Leibniz, diassociative, and dendriform cases holds for all seven algebras. Moreover, the dendriform case of Schafer's result follows similarly to the diassociative case after replacing $\dashv$ and $\vdash$ by $<$ and $>$ respectively, and replacing Lemma \ref{nilpotent equality} by the analogous Lemma \ref{dend nilpotent equality}.

The paper is structured as follows. For preliminaries, we define the relevant algebras and discuss notions of nilpotency. We state known lemmas concerning certain product algebras and briefly review extensions. We then derive Leibniz and diassociative analogues of the results found in \cite{yankosky}. We state the associative analogue as a corollary of the diassociative case. The final section of the paper contains several examples which highlight important intricacies in the results.

\section{Preliminaries}
Let $\F$ be a field. Throughout, all algebras will be $\F$-vector spaces equipped with bilinear multiplications which satisfy certain identities. First recall that a \textit{Leibniz algebra} $L$ is a nonassociative algebra with multiplication satisfying the \textit{Leibniz identity} $x(yz) = (xy)z + y(xz)$ for all $x,y,z\in L$.

\begin{defn}
A \textit{Zinbiel algebra} $Z$ is a nonassociative algebra with multiplication satisfying what we will call the \textit{Zinbiel identity} $(xy)z = x(yz) + x(zy)$ for all $x,y,z\in Z$.
\end{defn}

\begin{defn}
A \textit{diassociative algebra} (or \textit{associative dialgebra}) $D$ is a vector space equipped with two associative bilinear products $\dashv$ and $\vdash$ which satisfy the following identities for all $x,y,z\in D$:
\begin{enumerate}
    \item[D1.] $x\dashv (y\dashv z) = x\dashv (y\vdash z)$,
    \item[D2.] $(x\vdash y)\dashv z = x\vdash (y\dashv z)$,
    \item[D3.] $(x\dashv y)\vdash z = (x\vdash y)\vdash z$.
\end{enumerate}
\end{defn}

\begin{defn}
A \textit{dendriform algebra} $E$ is a vector space equipped with two bilinear products $<$ and $>$ which satisfy the following identities for all $x,y,z\in E$:
\begin{enumerate}
    \item[E1.] $(x<y)<z = x<(y<z) + x<(y>z)$,
    \item[E2.] $(x>y)<z = x>(y<z)$,
    \item[E3.] $(x<y)>z + (x>y)>z = x>(y>z)$.
\end{enumerate}
\end{defn}

The \textit{lower central series} is a well-known sequence of ideals that is defined recursively, for a Leibniz algebra $L$, by $L^0= L$ and $L^{k+1} = LL^k$ for $k\geq 0$. A Leibniz algebra is called \textit{nilpotent of class $u$}, denoted $\nil L = u$, if $L^u=0$ and $L^{u-1}\neq 0$ for some $u\geq 0$. The following lemma holds via induction and repeated application of the Leibniz identity.

\begin{lem}\label{left norming}
Let $L$ be a Leibniz algebra. Then $L^nL\subseteq LL^n$ for all $n$.
\end{lem}

For diassociative algebras, the definition of nilpotency is more involved. We take the following notions from \cite{di basri}. Let $A$ and $B$ be subsets of a diassociative algebra $D$ and define an ideal $A\lozenge B = A\dashv B + A\vdash B$ of $D$. There are notions of left, right, and general nilpotency for $D$ which are based on the $\lozenge$ operator. We define three sequences of ideals recursively for $k\geq 0$: \begin{itemize}
    \item[i.] $D^{\{0\}} = D$, $D^{\{k+1\}} = D\lozenge D^{\{k\}}$,
    \item[ii.] $D^{<0>} = D$, $D^{<k+1>} = D^{<k>} \lozenge D$,
    \item[iii.] $D^0 = D$, $D^{k+1} = D^0\lozenge D^k + D^1\lozenge D^{k-1} + \cdots + D^k\lozenge D^0$.
\end{itemize}

\begin{defn}
A diassociative algebra $D$ is called \begin{itemize}
    \item[i.] \textit{left nilpotent} if $D^{\{u\}} = 0$,
    \item[ii.] \textit{right nilpotent} if $D^{<u>} = 0$,
    \item[iii.] \textit{nilpotent} if $D^u = 0$
\end{itemize} for some $u\geq 0$. In particular, $D$ is \textit{nilpotent of class $u$} if $D^u=0$ and $D^{u-1}\neq 0$.
\end{defn} The following lemma from \cite{di basri} is crucial for the diassociative case in this paper.

\begin{lem}\label{nilpotent equality}
Let $D$ be a diassociative algebra. For all $n$, $D^{\{n\}} = D^{<n>} = D^n$.
\end{lem}

The same definitions may be stated for dendriform algebras with the simple substitutions of $<$ and $>$ for $\dashv$ and $\vdash$ respectively. The following lemma is the dendriform analogue of Lemma \ref{nilpotent equality}, and is proven in \cite{basri}.

\begin{lem}\label{dend nilpotent equality}
Let $E$ be a dendriform algebra. For all $n$, $E^{\{n\}} = E^{<n>} = E^n$.
\end{lem}

We now review extensions. Fix a type of algebra $\alg$ and let $A$ and $B$ be $\alg$ algebras. An \textit{extension} of $A$ by $B$ is a short exact sequence of the form $0\xrightarrow{} A\xrightarrow{\sigma} L\xrightarrow{\pi} B\xrightarrow{} 0$ where $L$ is a $\alg$ algebra and $\sigma$ and $\pi$ are \textit{homomorphisms}, i.e. linear maps which preserve the $\alg$ structure. An \textit{isomorphism} of $\alg$ algebras is a bijective homomorphism. A \textit{section} of the extension is a linear map $T:B\xrightarrow{} L$ such that $\pi T = \text{id}_B$.

\begin{defn}
An extension $0\xrightarrow{} A\xrightarrow{} L\xrightarrow{} B\xrightarrow{} 0$ of $A$ by $B$ is called \textit{nilpotent} if $L$ is nilpotent as an algebra.
\end{defn}

\section{Leibniz Case}
Consider a pair of nilpotent Leibniz algebras $A$ and $B$ and let $0\xrightarrow{} A\xrightarrow{\sigma} L\xrightarrow{\pi} B\xrightarrow{} 0$ be an extension of $A$ by $B$ with section $T:B\xrightarrow{} L$. We first define two ways for $B$ to act on $A$. Let $\vp:B\xrightarrow{} \Der(A)$ by $\vp(i)m = \sigma\inv(T(i)\sigma(m))$ and $\vp':B\xrightarrow{} \mathscr{L}(A)$ by $\vp'(i)m = \sigma\inv(\sigma(m)T(i))$ for $i\in B$, $m\in A$. Next, let \begin{align*}
    q&:\Der(A)\xrightarrow{} \Der(A)/\ad^l(A),\\ q'&:\mathscr{L}(A)\xrightarrow{} \mathscr{L}(A)/\ad^r(A)
\end{align*} denote the natural projections and define a pair of maps $(\Phi,\Phi') = (q\vp, q'\vp')$. We say that the pair $(\vp,\vp')$ is a \textit{lift} of $(\Phi,\Phi')$. Any two lifts $(\vp,\vp')$ and $(\psi,\psi')$ of $(\Phi,\Phi')$ are thus related by $\vp(i) = \psi(i) + \ad_{m_i}^l$ and $\vp'(i) = \psi'(i) + \ad_{m_i'}^r$ for $i\in B$ and some elements $m_i, m_i'\in A$ which depend on $i$. Our first proposition yields a criterion for when $L$ is nilpotent which is based on the following recursive construction. Define $A_0 = A$ and $A_{k+1} = \sigma\inv(\sigma(A_k)L + L\sigma(A_k))$ for $k\geq 0$.

\begin{prop}\label{prop 2.1}
Let $B$ be a nilpotent Leibniz algebra of class $s$. Then $L^{k+s} \subseteq \sigma(A_k) \subseteq L^k$ for all $k\geq 0$. Hence $L$ is nilpotent if and only if $A_k = 0$ for some $k$.
\end{prop}

\begin{proof}
Since $\pi:L\xrightarrow{} B$ is a homomorphism, one computes $\pi(L^s) = B^s = 0$, which implies that $L^s\subseteq \ker \pi = \sigma(A) = \sigma(A_0)$. Also, $\sigma(A_0) = \sigma(A) \subseteq L = L^0$. We therefore have a base case $L^s\subseteq \sigma(A_0)\subseteq L^0$ for $k=0$. Now suppose $L^{n+s} \subseteq \sigma(A_n)\subseteq L^n$ for some $n\geq 0$. Then \begin{align*}
    L^{n+1+s} &= LL^{n+s}\\ &\subseteq L\sigma(A_n) & \text{by induction} \\ &\subseteq \sigma(A_n)L + L\sigma(A_n)\\ &\subseteq L^nL + LL^n & \text{by induction} \\ &\overset{\ast}{=} LL^n \\ &= L^{n+1}
\end{align*} where $\sigma(A_n)L + L\sigma(A_n) =\sigma(A_{n+1})$ and the equality $\ast$ follows by Lemma \ref{left norming}. Thus $L^{s+k}\subseteq \sigma(A_k)\subseteq L^k$ for all $k\geq 0$ via induction. For the second statement, we first note that if $L$ is nilpotent, then $\sigma(A_k)\subseteq L^k = 0$ for some $k\geq 0$. This means $A_k=0$ since $\sigma$ is injective. Conversely, if $A_k=0$ for some $k\geq 0$, then $\sigma(A_k)=0$ and thus $L^{k+s}=0$. Hence $L$ is nilpotent.
\end{proof}

Again, let $(\vp,\vp')$ be a lift of $(\Phi,\Phi')$.

\begin{defn}
An ideal $N$ of $A$ is $(\vp,\vp')$\textit{-invariant} if $\vp(i)n, \vp'(i)n\in N$ for all $i\in B$ and $n\in N$.
\end{defn}

\begin{lem}
Let $(\vp,\vp')$ and $(\psi,\psi')$ be lifts of $(\Phi,\Phi')$. Then $N$ is $(\vp,\vp')$-invariant if and only if $N$ is $(\psi,\psi')$-invariant.
\end{lem}

\begin{proof}
Let $i\in B$. Since we have two lifts of the same pair, they are related by $\psi(i) = \vp(i) + \ad_{m_i}^l$ and $\psi'(i) = \vp'(i) + \ad_{m_i'}^r$ for some $m_i,m_i'\in A$. In one direction, assume $N$ is $(\vp,\vp')$-invariant. Then $\vp(i)n,\vp'(i)n\in N$ for all $n\in N$ by definition. Also, $m_in,nm_i'\in N$ for all $n\in N$ since $N$ is an ideal. Thus $\psi(i)n,\psi'(i)n\in N$ and so $N$ is $(\psi,\psi')$-invariant. The other direction is similar.
\end{proof}

\begin{defn}
An ideal $N$ of $A$ is $B$\textit{-invariant} if $N$ is $(\vp,\vp')$-invariant for some, and hence all, lifts of $(\Phi,\Phi')$.
\end{defn}

In particular, $A$ itself is $B$-invariant since $\vp(i),\vp'(i)\in \mathscr{L}(A)$ for all $i\in B$. Consider a $B$-invariant ideal $N$ of $A$ and let $(\vp,\vp')$ be a lift of $(\Phi,\Phi')$. We define $\G(N,\vp,\vp')$ to be the $B$-invariant ideal of $A$ generated by $AN$, $NA$, and $\{\vp(i)n,\vp'(i)n~|~ i\in B, n\in N\}$. Then $\G(N,\vp,\vp')\subseteq N$ and we reach the following lemma.

\begin{lem}
If $(\vp,\vp')$ and $(\psi,\psi')$ are lifts of $(\Phi,\Phi')$, then \[\G(N,\vp,\vp')= \G(N,\psi,\psi').\]
\end{lem}

\begin{proof}
It again suffices to show one direction. First note that $AN$ and $NA$ are contained in both sides of the equality by definition. For $i\in B$ and $n\in N$, we know $\psi(i)n = \vp(i)n + m_in$ and $\psi'(i)n = \vp'(i)n + nm_i'$ for some $m_i,m_i'\in A$. These expressions clearly fall in $\G(N,\vp,\vp')$ and therefore $\G(N,\psi,\psi')\subseteq \G(N,\vp,\vp')$.
\end{proof}

We now fix a lift $(\vp,\vp')$ of $(\Phi,\Phi')$ and denote $\G N = \G(N,\vp,\vp')$. Given $B$, $A$, \begin{align*}
    \Phi&:B\xrightarrow{} \Der(A)/\ad^l(A),\\ \Phi'&:B\xrightarrow{} \mathscr{L}(A)/\ad^r(A),
\end{align*} and a $B$-invariant ideal $N$ of $A$, define a descending sequence of $B$-invariant ideals $\G_k^B N$ of $N$ by $\G_0^B N = N$ and $\G_{k+1}^BN = \G(\G_k^B N)$ for $k\geq 0$.

\begin{thm}\label{thm 3.1}
Consider the extension $0\xrightarrow{} A\xrightarrow{\sigma} L\xrightarrow{} B\xrightarrow{} 0$ and our pair of maps $(\Phi,\Phi')$. If $A_0=A$ and $A_{k+1} = \sigma\inv(\sigma(A_k)L + L\sigma(A_k))$, then $A_k = \G_k^BA$ for all $k\geq 0$.
\end{thm}

\begin{proof}
By the work in \cite{mainellis}, there exists a unique factor system $(\vp,\vp',f)$ belonging to the extension $0\xrightarrow{} A\xrightarrow{\sigma} L\xrightarrow{} B\xrightarrow{} 0$ and $T$. By construction, $\vp$ and $\vp'$ are the maps of our lift $(\vp,\vp')$. Next, there exists another extension $0\xrightarrow{} A\xrightarrow{\iota} L_2\xrightarrow{} B\xrightarrow{} 0$ of $A$ by $B$ to which $(\vp,\vp',f)$ belongs. Here, $L_2$ is the vector space $A\oplus B$ equipped with multiplication $(m,i)(n,j) = (mn + \vp(i)n + \vp'(j)m + f(i,j),ij)$, where $f:B\times B\xrightarrow{} A$ is a bilinear form. Also $\iota(m) = (m,0)$. Since $(\vp,\vp',f)$ is equivalent to itself, the extensions are equivalent, and thus there exists an isomorphism $\tau:L\xrightarrow{} L_2$ such that $\tau\sigma = \iota$.

We will now prove the statement via induction, first noting that the base case $A_0 = A = \G_0^BA$ holds trivially. Assume that $A_n = \G_n^BA$ for some $n\geq 0$. By definition, it suffices to show the inclusion of generating elements for each side of the equality. Generating elements of $A_{n+1}$ have the forms $\sigma\inv(\sigma(m)x)$ and $\sigma\inv(x\sigma(m))$ for $x\in L$ and $m\in A_k$. Denote $\tau(x) = (m_x,i_x)\in L_2$. We compute \begin{align*}
    \sigma\inv(\sigma(m)x) &= \sigma\inv\tau\inv(\tau\sigma(m)\tau(x)) \\ &= \iota\inv((m,0)(m_x,i_x)) \\ &= \iota\inv(mm_x + \vp'(i_x)m,0) \\ &= mm_x + \vp'(i_x)m
\end{align*} and \begin{align*}
    \sigma\inv(x\sigma(m)) &= \sigma\inv\tau\inv(\tau(x)\tau\sigma(m)) \\ &= \iota\inv((m_x,i_x)(m,0)) \\ &= \iota\inv(m_xm + \vp(i_x)m,0) \\ &= m_xm + \vp(i_x)m.
\end{align*} Since $A_n = \G_n^BA$, one has $m_xm\in A(\G_n^BA)$ and $mm_x\in (\G_n^BA)A$, which are both included in $\G_{n+1}^BA$ since $\G_{n+1}^BA$ is the $B$-invariant ideal generated by $(\G_n^BA)A$, $A(\G_n^BA)$, and $\{\vp(i)m,\vp'(i)m~|~ m\in \G_n^BA,i\in B\}$. Thus $\vp'(i_x)m,\vp(i_x)m\in \G_{n+1}^BA$ as well and so $A_{n+1}\subseteq \G_{n+1}^BA$. Conversely, one computes \begin{align*}
    (\G_n^BA)A = \sigma\inv(\sigma(\G_n^BA)\sigma(A)) \subseteq \sigma\inv(\sigma(A_n)L) \subseteq A_{n+1},\\ A(\G_n^BA) = \sigma\inv(\sigma(A)\sigma(\G_n^BA))\subseteq \sigma\inv(L\sigma(A_n)) \subseteq A_{n+1}.
\end{align*} Also, let $i\in B$ and $m\in \G_n^BA = A_n$. Then $\vp(i)m = \sigma\inv(T(i)\sigma(m))\in A_{n+1}$ and $\vp'(i)m = \sigma\inv(\sigma(m)T(i))\in A_{n+1}$ since $T(i)\in L$. Therefore $\G_{n+1}^BA\subseteq A_{n+1}$.
\end{proof}

Given $B$, $A$, $\Phi:B\xrightarrow{} \Der(A)/\ad^l(A)$, and $\Phi':B\xrightarrow{} \mathscr{L}(A)/\ad^r(A)$, we define a new notion of nilpotency for $A$.

\begin{defn}
$A$ is $B$\textit{-nilpotent of class $u$}, written $\nil_B A = u$, if $\G_u^BA = 0$ and $\G_{u-1}^BA \neq 0$ for some $u\geq 0$.
\end{defn}
The following two corollaries hold similarly to the Lie case. For their proofs, simply replace Proposition 2.1 and Theorem 3.1 of \cite{yankosky} by the analogous Proposition \ref{prop 2.1} and Theorem \ref{thm 3.1} of the present paper. The subsequent theorem is the main result, which follows from these corollaries and the same logic as Yankosky's proof.

\begin{cor}
$L$ is nilpotent if and only if $B$ is nilpotent and $\G_u^B A = 0$ for some $u\geq 1$.
\end{cor}

\begin{cor}
$\max(\nil_BA,\nil B) \leq \nil L\leq \nil_B A + \nil B$.
\end{cor}

\begin{thm}
Let $(\vp,\vp')$ and $(\psi,\psi')$ be lifts of $(\Phi,\Phi')$ corresponding to extensions $0\xrightarrow{} A\xrightarrow{} L_{(\vp,\vp')}\xrightarrow{} B\xrightarrow{} 0$ and $0\xrightarrow{} A\xrightarrow{} L_{(\psi,\psi')}\xrightarrow{} B\xrightarrow{} 0$ respectively. Then $L_{(\vp,\vp')}$ is nilpotent if and only if $L_{(\psi,\psi')}$ is nilpotent.
\end{thm}

\section{Diassociative Case}
Consider a pair of nilpotent diassociative algebras $A$ and $B$ and an extension $0\xrightarrow{} A\xrightarrow{\sigma}L \xrightarrow{\pi} B\xrightarrow{} 0$ of $A$ by $B$ with section $T:B\xrightarrow{} L$. Throughout, we often let $*$ range over $\dashv$ and $\vdash$ for the sake of brevity. We consider four natural ways for $B$ to act on $A$. Define $\vp\dd, \vp\vv, \vp\dd',\vp\vv': B\xrightarrow{} \mathscr{L}(A)$ by $\vp_*(i)m = \sigma\inv(T(i)*\sigma(m))$ and $\vp_*'(i)m = \sigma\inv(\sigma(m)* T(i))$ for $i\in B$, $m\in A$. Let \begin{align*}
    q_*&:\mathscr{L}(A)\xrightarrow{} \mathscr{L}(A)/\ad_*^l(A),\\ q_*'&:\mathscr{L}(A)\xrightarrow{} \mathscr{L}(A)/\ad_*^r(A)
\end{align*} be the natural projections and define a tuple of maps $\Phi = (\Phi\dd,\Phi\vv,\Phi\dd',\Phi\vv')$ by $\Phi_* = q_*\vp_*$ and $\Phi_*' = q_*'\vp_*'$. The tuple $\vp = (\vp\dd,\vp\vv,\vp\dd',\vp\vv')$ is called a \textit{lift} of $\Phi$. Two lifts $\vp = (\vp\dd,\vp\vv,\vp\dd',\vp\vv')$ and $\psi = (\psi\dd,\psi\vv,\psi\dd',\psi\vv')$ of $\Phi$ are related by \begin{align*}
    \psi_*(i) &= \vp_*(i) + \ad_*^l(m_{*,i}), \\ \psi_*'(i) &= \vp_*'(i) + \ad_*^r(m_{*,i}')
\end{align*} for $i\in B$ and some $m_{*,i}, m_{*,i}'\in A$ which depend on $i$. Finally, let $A_0 = A$ and define $A_{k+1} = \sigma\inv(\sigma(A_k)\lozenge L + L\lozenge \sigma(A_k))$ for $k\geq 0$.

\begin{prop}
Let $B$ be a nilpotent diassociative algebra of class $s$. Then $L^{k+s} \subseteq \sigma(A_k)\subseteq L^k$ for all $k\geq 0$. Hence $L$ is nilpotent if and only if $A_k=0$ for some $k$.
\end{prop}

\begin{proof}
As with the Leibniz case, the base case $k=0$ follows by our definitions and the properties of extensions. Suppose $L^{n+s}\subseteq \sigma(A_n)\subseteq L^n$ for some $n\geq 0$. We recall that $L^n = L^{<n>} = L^{\{n\}}$ by Lemma \ref{nilpotent equality} and thereby compute \begin{align*}
    L^{n+1+s} &= L^{<n+1+s>}\\ &= L^{n+s}\lozenge L\\ &\subseteq \sigma(A_n)\lozenge L &\text{by induction} \\ &\subseteq \sigma(A_n)\lozenge L + L\lozenge\sigma(A_n) \\ &\subseteq L^{<n>}\lozenge L + L\lozenge L^{\{n\}} &\text{by induction} \\ &= L^{n+1}
\end{align*} where $\sigma(A_n)\lozenge L + L\lozenge\sigma(A_n) = \sigma(A_{n+1})$. Thus $L^{s+k} \subseteq \sigma(A_k)\subseteq L^k$ for $k\geq 0$ via induction. The second statement follows by the same logic as the Leibniz case.
\end{proof}

Once more, let $\vp = (\vp\dd,\vp\vv,\vp\dd',\vp\vv')$ be a lift of $\Phi$.

\begin{defn}
An ideal $N$ of $A$ is \textit{$\vp$-invariant} if $\vp_*(i)n,\vp_*'(i)n\in N$ for all $i\in B$, $n\in N$.
\end{defn}

\begin{lem}
Let $\vp$ and $\psi$ be lifts of $\Phi$. Then $N$ is $\vp$-invariant if and only if $N$ is $\psi$-invariant.
\end{lem}

\begin{proof}
Let $i\in B$. Since $\vp=(\vp\dd,\vp\vv,\vp\dd',\vp\vv')$ and $\psi=(\psi\dd,\psi\vv,\psi\dd',\psi\vv')$ are lifts of the same tuple, they are related by $\psi_*(i) = \vp_*(i) + \ad_*^l(m_{*,i})$ and $\psi_*'(i) = \vp_*'(i) + \ad_*^r(m_{*,i}')$ for some $m_{*,i},m_{*,i}'\in A$. In one direction, suppose $N$ is $\vp$-invariant. Then $\psi_*(i)n, \psi_*'(i)n\in N$ for all $n\in N$ since $N$ is a $\vp$-invariant ideal in $A$. Therefore $N$ is $\psi$-invariant. The converse is similar.
\end{proof}

\begin{defn}
An ideal $N$ of $A$ is \textit{$B$-invariant} if $N$ is $\vp$-invariant for some, and hence all, lifts of $\Phi$.
\end{defn}

In particular, $A$ is $B$-invariant since $\vp_*(i),\vp_*'(i)\in \mathscr{L}(A)$ for all $i\in B$. Now let $N$ be a $B$-invariant ideal in $A$ and $\vp$ be a lift of $\Phi$. We denote by $\G(N,\vp)$ the $B$-invariant ideal generated by $N\dashv A$, $N\vdash A$, $A\dashv N$, $A\vdash N$, and the set $\{\vp_*(i)n, \vp_*'(i)n~|~ i\in B,n\in N\}$. We thus have $\G(N,\vp)\subseteq N$ as well as the following lemma.

\begin{lem}
If $\vp$ and $\psi$ are lifts of $\Phi$, then $\G(N,\vp) = \G(N,\psi)$.
\end{lem}

\begin{proof}
It suffices to show that $\G(N,\psi)\subseteq \G(N,\vp)$. We first note that $N\dashv A$, $N\vdash A$, $A\dashv N$, and $A\vdash N$ are contained in both sides by definition. Similarly to the Leibniz case, the expressions for $\psi_*(i)n$ and $\psi_*'(i)n$ are clearly contained in $\G(N,\vp)$ for all $i\in B$ and $n\in N$. The converse holds without loss of generality.
\end{proof}

Fix a lift $\vp$ of $\Phi$ and denote $\G N = \G(N,\vp)$. Given $B$, $A$, $\Phi$, and a $B$-invariant ideal $N$ of $A$, define a descending sequence of $B$-invariant ideals $\G_k^BN$ of $N$ by $\G_0^BN := N$ and $\G_{k+1}^BN:= \G(\G_k^BN)$ for $k\geq 0$.

\begin{thm}
Consider $0\xrightarrow{} A\xrightarrow{\sigma} L\xrightarrow{} B\xrightarrow{} 0$ and let $\Phi$ be defined as above. If $A_0=A$ and $A_{k+1} = \sigma\inv(\sigma(A_k)\lozenge L + L\lozenge \sigma(A_k))$, then $A_k = \G_k^B A$ for all $k\geq 0$.
\end{thm}

\begin{proof}
As in the Leibniz case, the work with factor systems in \cite{mainellis} yields an equivalent extension $0\xrightarrow{} A\xrightarrow{\iota} L_2\xrightarrow{} B\xrightarrow{} 0$. Let $\tau:L\xrightarrow{} L_2$ be the equivalence. Here, $L_2$ is the vector space $A\oplus B$ equipped with multiplications $(m,i)*(n,j) = (m*n +\vp_*(i)n + \vp_*'(j)m + f_*(i,j),i*j)$, and $\iota(m) = (m,0)$. Moreover, $\vp_*$ and $\vp_*'$ are the same maps as in our lift $\vp$ while $f\dd$ and $f\vv$ are the bilinear forms in some factor system of diassociative algebras.

The base case of this result is trivial since $A_0 = A = \G_0^B A$ by definition. Now assume $A_n = \G_n^BA$ for some $n\geq 0$. Also by definition, it suffices to show the inclusion of generating elements for each side of the equality. Generating elements in $A_{n+1}$ have the forms $\sigma\inv(\sigma(m)*x)$ and $\sigma\inv(x*\sigma(m))$ for $m\in A_n$ and $x\in L$. Denote $\tau(x) = (m_x,i_x)\in L_2$. We compute \begin{align*}
    \sigma\inv(\sigma(m)*x) &= \sigma\inv\tau\inv(\tau\sigma(m)*\tau(x)) \\ &= \iota\inv((m,0)*(m_x,i_x)) \\ &= m*m_x + \vp_*'(i_x)m
\end{align*} and \begin{align*}
    \sigma\inv(x*\sigma(m)) &= \sigma\inv\tau\inv(\tau(x)*\tau\sigma(m)) \\ &= \iota\inv((m_x,i_x)*(m,0)) \\ &= m_x*m + \vp_*(i_x)m.
\end{align*} Since $A_n = \G_n^BA$, one has $m_x*m\in A*(\G_n^BA)$ and $m*m_x\in (\G_n^BA)*A$, which are included in $\G_{n+1}^BA$ since $\G_{n+1}^BA$ is the $B$-invariant ideal generated by $(\G_n^BA)*A$, $A*(\G_n^BA)$, and $\{\vp_*(i)m,\vp_*'(i)m~|~ m\in \G_n^BA,i\in B\}$. Thus $\vp_*'(i_x)m,\vp_*(i_x)m\in \G_{n+1}^BA$ as well. Therefore $A_{n+1}\subseteq \G_{n+1}^BA$. Conversely, one computes \begin{align*}
    (\G_n^BA)*A = \sigma\inv(\sigma(\G_n^BA)*\sigma(A)) \subseteq \sigma\inv(\sigma(A_n)*L) \subseteq A_{n+1},\\ A*(\G_n^BA) = \sigma\inv(\sigma(A)*\sigma(\G_n^BA))\subseteq \sigma\inv(L*\sigma(A_n)) \subseteq A_{n+1}.
\end{align*} Also, let $i\in B$ and $m\in \G_n^BA = A_n$. Then $\vp_*(i)m = \sigma\inv(T(i)*\sigma(m))\in A_{n+1}$ and $\vp_*'(i)m = \sigma\inv(\sigma(m)*T(i))\in A_{n+1}$ since $T(i)\in L$. Therefore $\G_{n+1}^BA\subseteq A_{n+1}$.
\end{proof}

\begin{defn}
Given $B$, $A$, and the tuple $\Phi$, we say that $A$ is $B$\textit{-nilpotent of class $u$}, written $\nil_B A = u$, if $\G_u^BA = 0$ but $\G_{u-1}^BA \neq 0$.
\end{defn}

The following results hold similarly to the Lie and Leibniz cases. Here, $\nil L$ is used to denote the nilpotency class of a diassociative algebra $L$.

\begin{cor}
$L$ is nilpotent if and only if $B$ is nilpotent and $\G_u^B A = 0$ for some $u\geq 1$.
\end{cor}

\begin{cor}
$\max(\nil_BA,\nil B) \leq \nil L\leq \nil_B A + \nil B$.
\end{cor}

\begin{thm}
Let $\vp$ and $\psi$ be lifts of $(\Phi\dd,\Phi\vv,\Phi\dd',\Phi\vv')$ corresponding to extensions $0\xrightarrow{} A\xrightarrow{} L_{\vp}\xrightarrow{} B\xrightarrow{} 0$ and $0\xrightarrow{} A\xrightarrow{} L_{\psi}\xrightarrow{} B\xrightarrow{} 0$ respectively. Then $L_{\vp}$ is nilpotent if and only if $L_{\psi}$ is nilpotent.
\end{thm}

We now state the associative case as a corollary since we have not been able to find it written down. Let $A$ and $B$ be associative algebras and consider a pair of maps $(\Phi,\Phi')$ such that $\Phi:B\xrightarrow{} \mathscr{L}(A)/\ad^l(A)$ and $\Phi':B\xrightarrow{} \mathscr{L}(A)/\ad^r(A)$. Let \textit{lifts} $(\vp,\vp')$ and $(\psi,\psi')$ of $(\Phi,\Phi')$ be defined similarly to the Leibniz case and consider their corresponding extensions $0\xrightarrow{} A\xrightarrow{} L_{(\vp,\vp')}\xrightarrow{} B\xrightarrow{} 0$ and $0\xrightarrow{} A\xrightarrow{} L_{(\psi,\psi')}\xrightarrow{} B\xrightarrow{} 0$ respectively. 

\begin{cor}
$L_{(\vp,\vp')}$ is nilpotent if and only if $L_{(\psi,\psi')}$ is nilpotent.
\end{cor}

\section{Examples}
The first two examples demonstrate that extensions corresponding to lifts of the same tuple need not have the same nilpotency class. We provide an example for the non-Lie Leibniz case as well as for the diassociative case.

\begin{ex}
Let $A=\langle x,y,z\rangle$ and $B=\langle w\rangle$ be abelian Leibniz algebras and consider two extensions $L_1$ and $L_2$ of $A$ by $B$. Let $L_1=\langle x,y,z,w\rangle$ have nonzero multiplications given by $w^2 = x$, $wx = y$, and $wy=z$. Then $L_1^2 = \langle x,y,z\rangle$, $L_1^3 = \langle y,z\rangle$, $L_1^4 = \langle z\rangle$, and $L_1^5=0$, making $L_1$ nilpotent of class 5. Now let $L_2=\langle x,y,z,w\rangle$ have nonzero multiplications given by $wx=y$ and $wy=z$. Then $L_2^2 = \langle y,z\rangle$, $L_2^3 = \langle z\rangle$, and $L_2^4 = 0$, making $L_2$ nilpotent of class $4$. Observe that $L_1$ and $L_2$ correspond to lifts of the same tuple, yet have different nilpotency classes. Indeed, $A$ is abelian, and hence $\ad^l(M)$ and $\ad^r(M)$ are zero, making $(\Phi,\Phi') = (\vp,\vp')$ for any lift of $(\Phi,\Phi')$. In this case, $\Phi(w)x = \vp(w)x= y$ and $\Phi(w)y= \vp(w)y=z$ for both. Also $\Phi'(w) = 0$.

We would also like to compute $A_k$ and $\G_k^BA$. Note that, since $A^2=0$, one needs only consider the actions of $\vp$ and $\vp'$ on $A$ when computing $\G_k^BA$. As predicted, $A_k=\G_k^BA$ for all $k$. One has \begin{align*}
    A_0 &= A = \G_0^BA,\\
    A_1 &= \langle y,z\rangle = \G_1^BA,\\
    A_2 &= \langle z\rangle = \G_2^BA,\\
    A_3 &= 0 = \G_3^BA,
\end{align*} and $A_k = 0 = \G_k^BA$ otherwise.
\end{ex}

\begin{ex}
Now for a diassociative example. Let $A=\langle x,y\rangle$ and $B=\langle u,v\rangle$ be abelian algebras and $L_{\vp}$ be an extension of $A$ by $B$ having nonzero multiplications $u\dashv u = x$, $u\vdash u = x+y$, $v\dashv v = y$, $v\vdash v = x+y$, and $v\vdash u = x+y = u\vdash v$. This diassociative algebra is a special case of the isomorphism type $Dias_4^1$ in Theorem 4.2 of \cite{di basri}. One computes $L_{\vp}^2 = \langle x,y\rangle$ and $L_{\vp}^3 = 0$; hence $L_{\vp}$ is nilpotent of class 3. We also note that the action of $B$ on $A$ is entirely zero, i.e. $\vp\dd = \vp\vv = \vp\dd' = \vp\vv' = 0$. Moreover, $A$ is again abelian, and hence all lifts of the natural $\Phi$ tuple are equal. To finish the point, the abelian extension $L_{ab}$ of $A$ by $B$ corresponds to the same zero-lift, but has nilpotency class 2.
\end{ex}

We conclude with an example in which $A$ is nonabelian and hence the lifts are allowed to vary by adjoint operators. In this example, however, our nilpotency classes turn out to be the same. We note that the algebras in this case are both associative and Leibniz.

\begin{ex}
Let $A=\langle x,y,z\rangle$ and $B=\langle w\rangle$ be associative algebras with only nonzero multiplications $x^2=y^2=z$. Consider two extensions $L_{(\vp,\vp')}$ and $L_{(\psi,\psi')}$ of $A$ by $B$. Let $L_{(\vp,\vp')}$ have nonzero multiplications given by $x^2=y^2=xw=z$, $wx=-z$ and let $L_{(\psi,\psi')}$ have nonzero multiplications given by $x^2=y^2=z$. These algebras are clearly nilpotent of class 3 since both have center $\langle z\rangle$ equal to their derived subalgebras. One computes $\vp(w)x = -z$, $\vp'(w)x = z$, and $\vp(w)y = \vp'(w)y = \vp(w)z = \vp'(w)z =0$. Also $\psi(w) = \psi'(w)=0$. Thus $\vp(w) = \psi(w) - \ad^l(x)$ and $\vp'(w) = \psi'(w) + \ad^r(x)$, and so we have lifts $(\vp,\vp')$ and $(\psi,\psi')$ that vary by adjoint operators.
\end{ex}

\section*{Acknowledgements}
The author would like to thank Ernest Stitzinger for the many helpful discussions.


\begin{thebibliography}{}


\bibitem{loday cup product} Loday, Jean-Louis. ``Cup-product for Leibniz Cohomology and Dual Leibniz Algebras.'' \textit{Mathematica Scandinavica}, Vol. 77, No. 2, 1995, pp. 189-196.

\bibitem{loday dialgebras} Loday, Jean-Louis. ``Dialgebras," in \textit{Dialgebras and related operads}, pp.7-66, Lecture Notes in Math., 1763, Springer, Berlin 2001.	arXiv:math/0102053.

\bibitem{mainellis} Mainellis, Erik. ``Nonabelian Extensions and Factor Systems for the Algebras of Loday.'' arXiv:2105.00116.

\bibitem{di basri} Rikhsiboev, Ikrom; Isamiddin Rakhimov; Witriany Basri. ``Four-Dimensional Nilpotent Diassociative Algebras.'' \textit{Journal of Generalized Lie Theory and Applications} Vol. 9, Issue 1, 2015.

\bibitem{basri} Rikhsiboev, Ikrom; Isamiddin Rakhimov; Witriany Basri. ``The Description of Dendriform Algebra Structures on Two-Dimensional Complex Space.'' \textit{Journal of Algebra, Number Theory: Advances and Applications} Vol. 4, No. 1, 2010.

\bibitem{schafer} Schafer, J. ``Extensions of Nilpotent Groups.'' \textit{Houston Journal of Mathematics} Vol. 21, No. 1, 1995. pp.1-16.

\bibitem{yankosky} Yankosky, Bill. ``On Nilpotent Extensions of Lie Algebras.'' \textit{Houston Journal of Mathematics} Vol. 27, No. 4, 2001.

\end{thebibliography}
\end{document}